\documentclass[11pt, reqno]{article}
\usepackage{amsmath}
\usepackage{amsthm}
\usepackage{amssymb}
\usepackage{latexsym}
\usepackage{hyperref}
\newtheorem{thm}{Theorem}[section]

\newtheorem{lem}{Lemma}[section]

\newtheorem{rem}{Remark}[section]
\newtheorem{ex}{Example}[section]
\numberwithin{equation}{section}

\begin{document}
\title{OPTIMAL CONTROL OF MARKOV PROCESSES WITH\\ AGE-DEPENDENT TRANSITION RATES
\thanks{This work is supported in part by SPM fellowship of
CSIR and in part by UGC Centre for Advanced Study.}}

\author{Mrinal K. Ghosh\thanks{
Department of Mathematics, Indian Institute of Science,
Bangalore-12, India, email: mkg@math.iisc.ernet.in} ,  Subhamay
Saha\thanks{ Department of Mathematics, Indian Institute of
Science, Bangalore-12, India, email: subhamay@math.iisc.ernet.in}}
\date{}

\maketitle \baselineskip20pt
\parskip10pt
\parindent.4in
\begin{abstract}
\noindent We study optimal control of Markov processes with
age-dependent transition rates. The control policy is chosen
continuously over time based on the state of the process and its
age. We study infinite horizon discounted cost and infinite horizon average cost problems.
Our approach is via the construction of an equivalent semi-Markov decision process. We
characterise the value function and optimal controls for both
discounted and average cost cases.
\end{abstract}


\section{Introduction}
We address optimal control of Markov processes in continuous time
taking values in a countable state space. The simplest example of
such a process is controlled Markov chains also known as
continuous time Markov decision process (CTMDP). The study of
controlled Markov chains is quite well developed \cite{yush},
\cite{guo}, \cite{hernandez}, \cite{pliska}; in particular see
\cite{hgbook} and the references therein. For a continuous time
controlled Markov chain, for each control input the holding time
or sojourn time in each state is exponentially distributed. Thus
for a fixed input the sojourn times are memoryless. If the sojourn
time in each state is given by a general distribution (other than
exponential) then the process is referred to as a semi-Markov
process. A controlled semi-Markov process, also known as
semi-Markov decision process(SMDP), is determined by a controlled
transition kernel and controlled holding time distributions. This
class of processes are usually studied via the embedded controlled
Markov chain \cite{fed1}, \cite{fed2}, \cite{wakuta}. Since in an
SMDP the holding time distributions have a memory, the age of the
process in a particular state influences the residual time in that
state. It may, however, be noted that the age has no influence in
determining the next state; nor does it play any role in the
decision making. There are several situations in which the age of
the process is crucial in the overall decision making process. To
illustrate this point we consider two examples.

 Consider a queueing system with controllable arrival and service rates.
 Suppose the queue capacity is infinite. The decision maker can dynamically
 select the service rates between the bounds $0< \mu_1 < \mu_2 < \infty$ depending
  on the number of persons in the queue and for how long that many persons are in the queue.
   Moreover, the arrival rates can also be adjusted between $0 < \gamma_1 < \gamma_2 < \infty.$
   The cost structure consists of three parts: a holding cost rate function $b(i,y)$ where $i$ is
   the number of customers and $y$ is the amount of time for which there has been $i$ customers,
   an income rate $b_1(\gamma)$ when an arrival rate $\gamma$ is maintained and a service cost
    rate $b_2(\mu)$ when the service rate is $\mu$. Mathematically the model can be described as below:
\begin{eqnarray*}S=\{0,1,2,\cdots\}:\; \mbox{state space}.\end{eqnarray*}
\begin{eqnarray*}U=[\gamma_1,\gamma_2]\times [\mu_1,\mu_2]: \; \mbox{control set}.\end{eqnarray*}
\begin{eqnarray*}\lambda_{ij}(y,\gamma, \mu)=\begin{cases}\gamma \; \mbox{for}\; j=i+1 \\
\mu \; \mbox{for}\; j=i-1\\
0 \; \mbox{otherwise}\end{cases}: \, \mbox{transition rates}.
\end{eqnarray*}
\begin{eqnarray*}c(i,y,\gamma, \mu)=b(i,y)-b_1(\gamma)+b_2(\mu):\; \mbox{cost function}.
\end{eqnarray*}

Next consider a device which is subject to shocks that occur
randomly in time according to a Poisson process with controllable
rate. Every shock causes a damage to the machine. The damage
caused depends on the state of the machine and the amount of time
it has been in that state. The machine can be in the states
$0,1,2, \cdots ,N.$ The state $0$ represents the new machine and
once the machine goes to state $N$, then a further shock would
mean that a new machine has to be installed. Suppose the rate of
arrival of shocks can be adjusted between  $0<
\mu_1<\mu_2<\infty.$ The cost structure consists of two parts: an
operational cost rate $b(i,y)$ is incurred if the machine is in
state $i$ and the age in that state is $y$, and a maintenance rate
$b_1(\mu)$ when the shock arrival rate is $\mu$. Mathematically,
the model can be described as below:
\begin{eqnarray*}S=\{0,1,\cdots,N\}\,.\end{eqnarray*}
\begin{eqnarray*}U=[\mu_1,\mu_2]\,.\end{eqnarray*}
\begin{eqnarray*}\lambda_{ij}(y,\mu)=\begin{cases}\frac{\mu}{1+y}\; \mbox{for}\; j=i+1,i\leq N-2\\
\frac{\mu y}{1+y} \; \mbox{for}\; j=i+2,i\leq N-2\\
\mu \; \mbox{for}\; i=N-1,j=N \; \mbox{and}\; i=N,j=0\;.\end{cases}
\end{eqnarray*}
\begin{eqnarray*}c(i,y,\mu)=b(i,y)+b_1(\mu)\;.
\end{eqnarray*}
Motivated by the above two examples we study optimal control of
Markov processes where the transition rates are age dependent.
Informally, this means if the process is in state $i$ and its age
in the state is $y$, then the probability that in an infinitesimal
time $dt$ the process will jump to state $j$ is
$\lambda_{ij}(y)dt$ plus a small error term. The probability that
after an infinitesimal time $dt$ it will still be in state $i$ is
$1- \displaystyle\sum_{j\neq i}\lambda_{ij}(y)dt$ plus some error
term, where $\lambda_{ij}$ are some measurable functions referred
to as transition rates. In controlled case the transition rates
also depends on the control parameter chosen dynamically based on
the state and the age. In continuous time  Markov chain the
transition rates are constant with respect to the age. In
semi-Markov case the transition rates are given by
$\lambda_{ij}(y)=p_{ij}\frac{f(y|i)}{1-F(y|i)}$, where $p_{ij}$s
are the transition probabilities and $F$ is the holding time
distributions with density $f$. In CTMDP and SMDP when the
controller is using a stationary control, he or she takes decision
only on the basis of state and it is independent of the age. But
in our case the decision maker takes his actions based on both the
state and the age. Thus the decision maker, unlike in CTMDP and
SMDP, has the liberty to take actions between jumps even when he
or she is using a stationary control. This liberty can be of great
advantage in practical situations. Hence our model may be more
effective in many practical situations.

We now present a formal description of the controlled process. A
rigorous construction of the process is given in the next section.
Let $S=\{0,1,2,\cdots \}$ be the state space and $U$ a compact
metric space, which is the control set.

For $i,j \in S$ with $i \neq j$ suppose $$\lambda_{ij}:
[0,\infty)\times U\rightarrow [0, \infty)$$ are given measurable
functions. Consider a controlled process $\{(X_t,Y_t)\}$ which
satisfies
\begin{eqnarray}\label{rate}\begin{cases}\mathbb{P}(X_{t+h}=j,Y_{t+h}=0 \,|\, X_t=i, Y_t=y,
 U_t = u )= \lambda_{ij}(y,u)h + o(h)
\\ \mathbb{P}(X_{t+h}=j,Y_{t+h}=y+h \,|\, X_t=i, Y_t=y,
U_t = u )= 1-\sum_{j\neq i}\lambda_{ij}(y,u)h + o(h)\,.
\end{cases}
\end{eqnarray}We call $\{X_t\}$ the state process, $\{Y_t\}$ the associated age
process and $\{U_t\}$ is the control process which is a $U$-valued
process satisfying certain technical conditions. The control
process is chosen based on both the state and its age. Thus the
control action is  taken continuously over time. Equation
\eqref{rate} implies that at time $t$ if the state is $i$, and its
age in the state is $y$ and the control chosen is $u$ then
$\lambda_{ij}(y,u)$ is the the infinitesimal jump rate to state
$j$.

The main aim in a stochastic optimal control problem is to find a
control policy which minimises a given cost functional. Let
$$c: S \times \mathbb{R}_+ \times U
\longrightarrow \mathbb{R}_+$$ be the running cost function.
Suppose the planning horizon is infinite and consider the
discounted cost problem. We seek to minimise
$$\mathbb{E}\int_0^{\infty}e^{-\alpha t}
c(X_t,Y_t,U_t)dt$$ over the set of all admissible controls (to be
defined in the next section), where $\alpha > 0$ is the discount
factor. We also study the long-run average cost on the infinite
horizon.

We now briefly comment on some earlier work leading to ours. Hordijk et al.
\cite{hord1,hord2,hord3} have studied Markov drift decision
processes which is an important  generalisation of semi-Markov decision processes. However,
in their work though the state drifts according to a specified drift function between jumps,
no action is taken during the period. There is another important class of controlled processes
 namely piecewise deterministic processes(PDP) \cite{davis}, where decisions are taken between jumps as well.
 But in PDP the importance of age has not been emphasized.

 \noindent The rest of the paper is structured as follows. In
Section $2$ we use the idea in \cite{anind} to give a rigorous
construction of the process $\{(X_t, Y_t)\}$ which is based on  a
representation of $\{(X_t, Y_t)\}$ as stochastic integrals with
respect to an appropriate Poisson random measure. In Section $3$
we study the infinite horizon discounted cost problem. For that we
construct an equivalent semi-Markov process. Section $4$ deals
with the infinite horizon average cost case.

\section{Mathematical Model and Preliminaries}
Let $(\Omega, \mathcal{F}, \mathbb{P})$ be the underlying
probability space. For $i,j \in S,i \neq j$, let
$$\lambda_{ij}: [0,\infty) \times U \rightarrow
[0,\infty)$$ be given measurable functions. Set
$$\lambda_{ii}(y,u)= - \sum_{j \neq i}\lambda_{ij}(y,u) \;.$$
We make the following assumption which is in force throughout this
paper:

\noindent \textbf{(A1)}$\,\,\,\,$ There exists a constant $M$ such that $$\displaystyle \sup_{i \in
S,y\geq 0,u \in U}\{-\lambda_{ii}(y,u)\} < M \;.$$

\noindent \textbf{(A2)}$\,\,\,\,$ $\displaystyle \inf_{i \in S,y\geq 0,u \in U}\{-\lambda_{ii}(y,u)\} > m $ for some $m>0$.

For technical reasons we introduce  relaxed control framework.
 Let $\mathcal{P}(U)$ denote the set of
probability measures on $U$.  For $i \neq j$, let
$\tilde{\lambda}_{ij}:[0,\infty)\times \mathcal{P}(U) \rightarrow
\mathbb{R}_+$ be defined by
$$\tilde{\lambda}_{ij}(y,\nu)=\int_U \lambda_{ij}(y,u)\nu (du)\,.$$

\noindent For $i\neq j$, $y \in \mathbb{R}_+$ and $\nu \in
\mathcal{P}(U)$, let $\Lambda_{ij}(y,\nu)$ be consecutive right open, left closed
intervals of the real line of length $\tilde{\lambda}_{ij}(y,\nu)$.

\noindent We define a function $h: S \times \mathbb{R}_+ \times
\mathcal{P}(U) \times \mathbb{R} \rightarrow \mathbb{R}$ by
\begin{eqnarray}\label{h}
h(i,y,\nu,z)=\begin{cases}
j-i~ &\mbox{if} ~ z \in \Lambda_{ij}(y,\nu)\\
0 &\text{otherwise}\,.
\end{cases}
\end{eqnarray}
We also define a function $g: S \times \mathbb{R}_+ \times
\mathcal{P}(U) \times \mathbb{R} \rightarrow \mathbb{R}$ by
\begin{eqnarray}\label{g}
g(i,y,\nu,z)=\begin{cases}
y ~ &\mbox{if}~ z \in \displaystyle \bigcup_j \Lambda_{ij}(y,\nu)\\
0 &\text{otherwise}\,.
\end{cases}
\end{eqnarray}
Let $\wp(ds,dz)$ be a Poisson random measure on $\mathbb{R}_+
\times \mathbb{R} $ with intensity measure $ds\times dz$, the
product Lebesgue measure on $\mathbb{R}_+ \times \mathbb{R}.$

\noindent Consider the following stochastic differential equation
\begin{eqnarray}\label{proc} \begin{cases}X_t = X_0 + \int_0^t \int_ {\mathbb{R}}h_1(X_{s-}, Y_{s-},U_s, z)\wp(ds,dz)
\\ Y_t = Y_0 + t- \int_0^t \int_ {\mathbb{R}} h_2(X_{s-},Y_{s-},U_s,z) \wp(ds,dz)
\end{cases}
\end{eqnarray}
where $\{U_t\}$ is a $\mathcal{P}(U)$-valued process with
measurable sample paths which is predictable with respect to the filtration given by
$$\sigma(\wp(A \times B):A \in \mathcal{B}([0,s]), B \in \mathcal{B}(\mathbb{R}),s\leq t)$$ and
$X_0$, $Y_0$ are random variables with prescribed laws independent
of the Poisson random measure. The integrals in $\eqref{proc}$ are
over $(0,t]$. From the results in \cite[Chap IV, p.
231]{ikeda1989stochastic} it follows that for each $\{U_t\}$ as
above, equation $\eqref{proc}$ has an a.s unique strong solution
$\{(X_t,Y_t)\}$.
 If $U_t= \textbf{u}(t,X_{t-}, Y_{t-})$ for some measurable function $\textbf{u}:[0,\infty)\times S \times [0,\infty)\rightarrow \mathcal{P}(U)$
 then $U$ is called a Markov control. Moreover if $U_t= \textbf{u}(X_{t-}, Y_{t-})$
 for some measurable function $\textbf{u}:S \times [0,\infty) \rightarrow \mathcal{P}(U)$ then $U$ is referred to as stationary Markov control.
 It is customary in optimal control literature to refer to the function $\textbf{u}$ as the control.
 We denote by $\mathcal{U}$ the set of all measurable functions $\textbf{u}:S \times [0,\infty)\rightarrow \mathcal{P}(U)$.
In this paper we restrict our set of controls to the set
$\mathcal{U}$ and we refer to $\mathcal{U}$ as the set of
admissible controls. For each $\textbf{u} \in \mathcal{U}$,
$\{(X_t,Y_t)\}$ is a strong Markov process. Let $f: S \times
\mathbb{R}_+ \rightarrow \mathbb{R}$ be continuously
differentiable in the second variable. Then applying It\^{o}'s
formula to $f$ we can show that the generator of the process
$\{(X_t, Y_t)\}$ denoted by $\mathcal{A}^{\textbf{u}}$ is given by
\begin{eqnarray}\mathcal{A}^{\textbf{u}}f(i,y)= \frac{\partial f}{\partial y}(i,y)+
 \sum_{j \neq i}\tilde{\lambda}_{ij}(t,y,\textbf{u}(i,y))[f(j,0)-f(i,y)] \;.
\end{eqnarray}
\section{Infinite Horizon Discounted case}
Let
$$c: S \times \mathbb{R}_+ \times U
\longrightarrow \mathbb{R}_+$$ be the running cost function.
Define $\tilde{c}: S \times \mathbb{R}_+ \times \mathcal{P}(U)
\longrightarrow \mathbb{R}_+$ by $$ \tilde{c}(i, y, \nu) = \int_U
c(i, y, u) \nu (du).$$
Let $\alpha
> 0$ be the discount factor. Then for $\textbf{u} \in \mathcal{U}$
the infinite horizon discounted cost is given by
\begin{eqnarray}J^{\textbf{u}}_{\alpha}(i)= E_{i,0}^{\textbf{u}}\int_0^{\infty}e^{-\alpha t}\tilde{c}(X_t,Y_t,\textbf{u}(X_t,Y_t))dt
\end{eqnarray}
where $E^{\textbf{u}}_{i,0}$ denotes the expectation when the
control $\textbf{u}$ is used and $X_0 = i, Y_0 = 0$. The objective
is to minimise $J^{\textbf{u}}_{\alpha}(i)$ over all admissible
controls. So we define
\begin{eqnarray}V_{\alpha}(i)= \inf_{\textbf{u} \in \mathcal{U}}J^{\textbf{u}}_{\alpha}(i)\,.
\end{eqnarray}
The function $V_{\alpha}$ is called the ($\alpha$-discounted)
value function. An admissible control $\textbf{u}^* \in
\mathcal{U}$ is called ($\alpha$-discounted) optimal  if
$$ J^{\textbf{u}^*}_{\alpha}(i) = V_{\alpha}(i).$$
 We carry out our study under the following assumptions :

\noindent \textbf{(A3)}$\,\,\,\,$ $\lambda_{ij}$s $(j \neq i)$ are jointly continuous in $y$ and $u$ and
the sum $\displaystyle \sum_{j \neq i}\lambda_{ij}(y,u)$ converges uniformly for each $i$.

\noindent \textbf{(A4)}$\,\,\,\,$ The cost function $c$ is continuous in the second and third variable
and there exists a finite constant $\tilde{C}$ such that
$$\sup_{i,y,u}c(i,y,u) \leq \tilde{C}\,.$$

\noindent The boundedness of $c$ implies that $V_{\alpha}(i)$ is
well defined for each $i$ and
\begin{eqnarray*}\sup_{i}V_{\alpha}(i)\leq \frac{\tilde{C}}{\alpha}\;.
\end{eqnarray*} In order to characterise the value function and the optimal control we construct an equivalent semi-Markov decision process.
In order to do so the key observation here is that between jumps the trajectory of the
process $\{(X_t,Y_t)\}$ is deterministic. Thus $\{(X_t,Y_t)\}$ is
a piecewise deterministic process \cite{davis}. Therefore a
stationary relaxed control is equivalent to that of choosing a
function $r:[0,\infty) \longrightarrow \mathcal{P}(U)$ at each
jump time. More explicitly suppose the process jumps to a state
$(i,0)$, then we choose the function $r_i$ given by
$r_i(y)=\textbf{u}(i,y)$.

\noindent Let
$$\mathcal{R}=\{r\,|\,r:[0,\infty) \longrightarrow \mathcal{P}(U), \,measurable\}$$
This set $\mathcal{R}$ will be the action space for an equivalent
semi-Markov decision process that we are going to construct. First
we give a topology on $\mathcal{R}$. Let $V=L^1([0,\infty);C(U))$,
where $C(U)$ is space of continuous functions on $U$ endowed with
the supremum norm.. Thus $V$ is the space of integrable (with
respect to Lebesgue measure) $C(U)$-valued functions on
$[0,\infty)$. Then the dual of $V$ is
$V^*=L^{\infty}([0,\infty);M(U))$, where $M(U)$ is the space of
complex Borel regular measures on $U$ with the total variation
norm. Now by Banach-Alaoglu theorem the unit ball of $V^*$ is
weak$^*$ compact. Hence $\mathcal{R}$ being a closed subset of the
unit ball of $V^*$, is a compact metric space (for more details
see \cite[Chap 4, p. 149]{davis}).  In this topology, $r^n
\longrightarrow r$ if and only if
$$\int_{[0,\infty)}\int_{U}f(y,u)r^n_y(du)dy \longrightarrow \int_{[0,\infty)}\int_{U}f(y,u)r_y(du)dy$$for all $f\in V$.

\noindent Now define $f:S\times \mathcal{R}\longrightarrow \mathbb{R}_+$ by
\begin{eqnarray}\label{expcost}f(i,r)= \int_0^{\infty}\biggl(\exp(-\alpha y)\exp\{- \int_0^y\int_U\sum_{k \neq i}
\lambda_{ik}(s,u)r_s(du)ds\}\int_Uc(i,y,u)r_y(du)\biggr)dy\,.
\end{eqnarray}
For $r \in \mathcal{R}$ define a transition matrix by
\begin{eqnarray}\label{tranprob}\hat{p}_{ij}(r)=\int_0^{\infty}\biggl(\exp\{- \int_0^y\int_U
\sum_{k \neq i}\lambda_{ik}(s,u)r_s(du)ds\}\int_U\lambda_{ij}(y,u)r_y(du)\biggr)dy\,.
\end{eqnarray}
Finally for $r \in \mathcal{R}$ and $ t \in \mathbb{R}_+$ define a family of distribution functions by
\begin{eqnarray}\label{sojdist}F^r_{ij}(t)=\frac{\int_0^{t}\biggl(\exp\{- \int_0^y\int_U\displaystyle
\sum_{k \neq i}\lambda_{ik}(s,u)r_s(du)ds\}\int_U\lambda_{ij}(y,u)r_y(du)\biggr)dy}{\hat{p}_{ij}(r)}\,.
\end{eqnarray}
Now consider a semi-Markov decision process with state space $S$,
action space $\mathcal{R}$, expected one stage cost $f$ given by
$\eqref{expcost}$, transition probabilities $(\hat{p}_{ij}(r))$
given by $\eqref{tranprob}$ and sojourn time distributions
$F_{ij}^r$ given by \eqref{sojdist}. In short the dynamics of the
process is as follows: Suppose the initial state is $i \in S$ and
the decision maker chooses an action $r$ from the set
$\mathcal{R}$. The action depends on the state. Because of this
action the decision maker has to pay a cost up to the next jump
time at a rate dependent on the state and the action chosen. The
next state is $j$ with probability $\hat{p}_{ij}(r)$ and
conditioned on the event that the next state is $j$, the
distribution of the sojourn time in the state $i$ is given by
$F^r_{ij}$. The aim of the decision maker is to minimize the cost
over the set of stationary policies $\pi:S\longrightarrow
\mathcal{R}$.

\noindent Define
\begin{eqnarray}\widetilde{J}_{\alpha}^{\pi}(i)=\mathbb{E}^{\pi}_{i}\displaystyle
 \sum_{n=0}^{\infty}e^{-\alpha(\tau_0+\tau_1+\cdots+\tau_n)}\int_0^{\tau_{n+1}}e^{-\alpha y}\biggl(\int_U c(X_{T_n},y,u)\pi_{X_{T_n}}(y)(du)\biggr)dy
\end{eqnarray}
where $T_n$ is the $n$th jump time and $\tau_n=T_n-T_{n-1}$. Let
 $$\tilde{V}_{\alpha}(i)=\displaystyle \inf_{\pi}\widetilde{J}_{\alpha}^{\pi}(i).$$ Thus $\tilde{V}_{\alpha}$ is the value function for the SMDP.
\noindent Now corresponding to a control $\textbf{u}$ of the original optimal
control problem, define the policy $\pi^{\textbf{u}}$ for the semi-Markov decision process by
$$\pi^{\textbf{u}}_i(y)=\textbf{u}(i,y).$$ Then it follows from the definition of the semi-Markov
decision process that
\begin{align*}J^{\textbf{u}}_{\alpha}(i) &= \mathbb{E}^{\textbf{u}}_{i,0}\bigl[\displaystyle\sum_{n=0}^{\infty}
\int_{T_n}^{T_{n+1}}e^{-\alpha t}\tilde{c}(X_t,Y_t,\textbf{u}(X_t,Y_t))dt\bigr]\\
&= \sum_{n=0}^{\infty}\mathbb{E}^{\textbf{u}}_{i,0} \bigl[\mathbb{E}^{\textbf{u}}_{i,0}\bigl[ \int_{T_n}^{T_{n+1}}e^{-\alpha t}\tilde{c}(X_t,Y_t,\textbf{u}(X_t,Y_t))dt|H_n\bigr]\bigr]\\
&=\mathbb{E}^{\pi^{\textbf{u}}}_{i}\displaystyle
 \sum_{n=0}^{\infty}e^{-\alpha(\tau_0+\tau_1+\cdots+\tau_n)}\int_0^{\tau_{n+1}}e^{-\alpha y}\biggl(\int_U c(X_{T_n},y,u)\pi_{X_{T_n}}^{\textbf{u}}(y)(du)\biggr)dy\\
 &=\widetilde{J}_{\alpha}^{\pi^{\textbf{u}}}(i)
\end{align*} where $H_n$ is the history upto the nth jump time.
On the other hand corresponding to a policy $\pi$ of
the SMDP define the control $\textbf{u}^{\pi}$ for the original optimal control problem by
$$\textbf{u}^{\pi}(i,y)=\pi_i(y).$$ Again
$$J^{\textbf{u}^{\pi}}_{\alpha}(i) = \tilde{J}_{\alpha}^{\pi}(i).$$
Hence it follows that
 \begin{eqnarray}\label{equi}V_{\alpha}(i) = \widetilde{V}_{\alpha}(i).\end{eqnarray}
 The equation $\eqref{equi}$ establishes the equivalence between the original control problem and the constructed semi-Markov
 decision process.

 Thus in order to evaluate $V_{\alpha}(i)$,
 we analyse the the equivalent semi-Markov decision process. As a first step we state the following useful lemma.
\begin{lem} Under (A1) - (A4), the functions $f(i,.)$, $\hat{p}_{ij}(.)$ and $F_{ij}^{(.)}(t_0)$ are continuous on $\mathcal{R}$.
\end{lem}
\begin{proof}Suppose $r^n$ converges to $r$ in $\mathcal{R}$. Then
\begin{align*}|f(i,r_n)-f(i,r)| &\leq \tilde{C}\int_0^{\infty}e^{-\alpha t}\bigl|
e^{- \int_0^t\int_U\sum_{k \neq i}\lambda_{ik}(s,u)r^n_s(du)ds}-e^{- \int_0^t\int_U\sum_{k \neq i}\lambda_{ik}(s,u)r_s(du)ds}\bigr|dt\\
&+\biggl|\int_0^{\infty}\int_Ue^{-\alpha t}e^{- \int_0^t\int_U\sum_{k \neq i}\lambda_{ik}(s,u)r_s(du)ds}c(i,t,u)r_t^n(du)dt-\\
&\quad \int_0^{\infty}\int_Ue^{-\alpha t}e^{- \int_0^t\int_U\sum_{k \neq i}\lambda_{ik}(s,u)r_s(du)ds}c(i,t,u)r_t(du)dt\biggr|\,.
\end{align*}By the definition of convergence in $\mathcal{R}$, both the terms on the right hand side of the above expression go to $0$ as $n\rightarrow \infty$.
Similar arguments hold for the other two functions as well.
\end{proof}
Thus using the equivalence of the semi-Markov decision process
described above and the original control problem, we obtain the
following result from the standard theory of SMDP
\cite{ross1992applied}.
\begin{thm}Assume (A1) - (A4). Then the value function $V_{\alpha}$ is the unique bounded solution of
\begin{eqnarray}\label{esm}\phi(i)=\min_{r \in \mathcal{R}}\bigl[f(i,r)+\sum_{j \neq i}
\hat{p}_{ij}(r)\int_0^{\infty}e^{-\alpha t}\phi(j)dF_{ij}^{r}(t)\bigr]
\end{eqnarray}Furthermore if $r^*_{i}$ is the minimizer of the right hand
side of $\eqref{esm}$ (which exists by the previous lemma and
compactness of $\mathcal{R}$), then the control given by
$\textbf{u}^*(i,y)=r_i^*(y)$ is an optimal control for the
original control problem.
\end{thm}
\begin{rem} The reason for restricting to only stationary controls is evident from our approach.
For setting a bijection between the set of controls of the
original control problem and the equivalent SMDP, we need the
restriction on the set of admissible controls. For a Markov
control it is not clear that such a bijection can be established.
Because in CTMDP as well as in SMDP, the optimal control is
finally given by a stationary control, this restriction is not
unnatural.
\end{rem}
\section{Infinite Horizon Average Cost}
Now we investigate the infinite horizon average cost cost problem via the equivalent semi-Markov decision process approach.
First we describe the infinite horizon average cost control problem for the original control problem.
For $\textbf{u} \in \mathcal{U}$ define
$$J^{\textbf{u}}(i)=\displaystyle \limsup_{n \rightarrow \infty}\frac{\mathbb{E}_{i,0}^{\textbf{u}}
\int_0^{T_n} \int_Uc(X_t,Y_t,u)\textbf{u}(X_t,Y_t)(du)dt}{\mathbb{E}_{i,0}^{\textbf{u}}T_n}\,,$$
where $T_n$ is the $n$th jump time.
The aim of the controller is to minimise $J^{\textbf{u}}$ over all $\textbf{u}.$

\noindent Now consider the semi-Markov decision process defined in
the previous section with the expected one-stage (jump to jump)
cost in state $i$ given by
\begin{eqnarray*}\varphi(i,r)=\int_0^{\infty}\biggl(\exp\bigl\{- \int_0^y\int_U\sum_{k \neq i}
\lambda_{ik}(s,u)r_s(du)ds\bigr\}\int_Uc(i,y,u)r_y(du)\biggr)dy\,.
\end{eqnarray*}where $r \in \mathcal{R}$ is the action chosen in state $i$.

\noindent Now define
$$\tilde{J}^{\pi}(i)=\displaystyle \limsup_{n \rightarrow \infty}\frac{\mathbb{E}^{\pi}_iZ(T_n)}{\mathbb{E}_i^{\pi}(T_n)}\, ,$$
where $$Z(T_n)=\displaystyle \sum_{k=0}^{n-1}\int_{0}^{\tau_{k+1}}\int_{U}c(X_{T_k},y,u)\pi_{X_{T_k}}^{}(y)(du)dy$$ is the cost incurred up to the $n$th jump time.

By arguments analogous to the discounted case we have
$$\inf_{\textbf{u}}J^{\textbf{u}}(i)=\inf_{\pi}\tilde{J}^{\pi}(i).$$
Let $\bar{\tau}(i,r)$ be the expected sojourn time of the
equivalent semi-Markov decision process in state i, when the
action chosen is $r$. Thus
\begin{eqnarray*}\bar{\tau}(i,r)=\int_0^{\infty}\exp\{-\int_0^t\int_U\displaystyle \sum_{j \neq i} \lambda_{ij}(y,u)r_y(du)dy\}dt\, .\end{eqnarray*}
Consider the equation
\begin{eqnarray}\label{smdp}\psi(i)=\inf_{r \in \mathcal{R}}[\varphi(i,r)+\sum_{j \neq i}\hat{p}_{ij}(r)\psi(j)-\rho \bar{\tau}(i,r)]
\end{eqnarray}
where $\psi:S \rightarrow \mathbb{R}$ and $\rho$ is a scalar.

Using the equivalence and the theory of SMDP
{\cite{ross1992applied}}, we obtain the following result:
\begin{thm}If $\eqref{smdp}$ has a solution $(h,g)$, where $h$ is a bounded
function,
then $g$ is the optimal average cost for the original control
problem and an optimal policy is given by
$\textbf{u}^*(i,y)=r_i^*(y)$ where $r_i^*$ is given by
{\small\begin{eqnarray*} [\varphi(i,r^*_i)+\sum_{j \neq
i}\hat{p}_{ij}(r^*_i)h(j)-g \bar{\tau}(i,r^*_i)]=\inf_{r \in
\mathcal{R}}[\varphi(i,r)+\sum_{j \neq
i}\hat{p}_{ij}(r)h(j)-g\bar{\tau}(i,r)]\,.
\end{eqnarray*}}
\end{thm}
Now we give conditions explicit conditions on $\lambda_{ij}$ which
will ensure the existence of a bounded solution of $\eqref{smdp}$.
We make two additional assumptions:

\noindent \textbf{(A5)}\,\,\,\, $S$ is a finite set.

\noindent \textbf{(A6)}\,\,\,\, The exists $\delta>0$ such that $\lambda_{i0}(y,u)> \delta$ for all
$i(\neq 0), y , u$ and for $j\neq 0$ if $\displaystyle \sup_{y,u}\lambda_{ij}(y,u)>0$, then $\displaystyle \inf_{y,u}\lambda_{ij}(y,u)>0$.
\begin{rem} Note that even though $S$ is finite, the effective state space is $S \times \mathbb{R}_+$ which is uncountable.
\end{rem}

\noindent Now we give an example where our assumptions are true.

\begin{ex} We modify the second example in the introduction. Let
$\lambda_{ij}$ be modified as follows:

\noindent For $i \leq N-3$,
\begin{eqnarray*}\lambda_{iN}(y,\mu)=\frac{\mu}{10^{N-i}}\,.
\end{eqnarray*}
\begin{eqnarray*}\lambda_{ii+1}(y,\mu)=\begin{cases}\mu-\frac{2\mu}{10^{N-i}}\quad \mbox{for}\;y\leq 100 \\
\frac{\mu}{10^{N-i}}\quad \mbox{for}\;y\geq 1000 \\
\mbox{linear\, in\, between}\,.
\end{cases}
\end{eqnarray*}
\begin{eqnarray*}\lambda_{ii+2}(y,\mu)=\mu-\frac{\mu}{10^{N-i}}-\lambda_{ii+1}(y,\mu)\,.
\end{eqnarray*}
\begin{eqnarray*}\lambda_{N-2N-1}(y,\mu)=\begin{cases}\mu-\frac{2\mu}{10^{2}}\quad \mbox{for}\;y\leq 100 \\
\frac{2\mu}{10^{2}}\quad \mbox{for}\;y\geq 1000 \\
\mbox{linear\, in\, between}\,.
\end{cases}
\end{eqnarray*}
\begin{eqnarray*}\lambda_{N-2N}(y,\mu)=\mu-\lambda_{N-2N-1}(y,\mu)\,.\end{eqnarray*}
\begin{eqnarray*}\lambda_{N-1N}(y,\mu)=\mu\,.\end{eqnarray*}
\begin{eqnarray*}\lambda_{N0}(y,\mu)=\mu\,.\end{eqnarray*}
Clearly this example satisfies (A5) and (A6) with $N$ playing the role of $0$.
\end{ex}

\noindent For $\textbf{u} \in \mathcal{U}$ it follows from (A6) that the transition probabilities of the embedded Markov
chain $\{X_{T_n}\}$ where $T_n$ are the successive jump times, satisfy:
\begin{align*}\hat{p}_{i0}^{\textbf{u}}&= \int_0^{\infty}\tilde{\lambda}_{i0}(y, \textbf{u}(i,y))
\exp\bigl(-\int_0^y\displaystyle \sum_{j \neq i}\tilde{\lambda}_{ij}(s,\textbf{u}(i,s))ds\bigr)dy \\
& \geq \delta \int_0^{\infty}\exp(-My)dy \\
&= \frac{\delta}{M}\;.
\end{align*}
This implies that in the embedded Markov chain, the expected
number of steps taken to reach $0$ starting from any state $i$ is
finite, i.e., if
$$N=\min\{n\geq 1|X_{T_n}=0\}$$ then
\begin{eqnarray}\label{frt}\displaystyle \sup_{\textbf{u}\in \mathcal{U}}\mathbb{E}_i^{\textbf{u}}N <\infty \,. \end{eqnarray}
Also by (A6) it follows that if $\hat{p}^{\textbf{u}}_{ij}\neq 0$
then $\displaystyle \inf_{\textbf{u}\in
\mathcal{U}}\hat{p}^{\textbf{u}}_{ij}>0$.

Let
\begin{eqnarray}\label{firstzero}\tau_0=\inf\{t>0|(X_t,Y_t)=(0,0)\}\,.\end{eqnarray}
\begin{lem}Under (A1)-(A3), (A5)-(A6) we have
\begin{eqnarray}\displaystyle\sup_{\textup{\textbf{u}}}\mathbb{E}^{\textup{\textbf{u}}}_{i,0}\tau_0 < \infty ,\end{eqnarray}
where $\tau_0$ is as in $\eqref{firstzero}$.
\end{lem}
\begin{proof} Let $\delta_n$ denote the set of sequences of states $(i_0,i_1, \cdots,i_n)$ such that
$$i_0=i$$
$$i_j \neq 0 \quad \mbox{for}\quad j=1,2, \cdots, n-1$$
$$i_n=0.$$
Then \begin{eqnarray*}\mathbb{E}^{\textbf{u}}_{i,0}\tau_0= \sum_{n=1}^{\infty} \sum_{(i_0,i_1,
\cdots,i_n)\in \delta_n}\prod \hat{p}^{\textbf{u}}_{i_k, i_{k+1}}(\eta_{i_0i_1}^{\textbf{u}}+ \cdots + \eta_{i_{n-1}i_n}^{\textbf{u}})
\end{eqnarray*}
where $\eta_{ij}^{\textbf{u}}$ is the expected amount of time spent in state $i$ given that the next transition will be into state $j$.
Therefore \begin{eqnarray*}\mathbb{E}^{\textbf{u}}_{i,0}\tau_0 \leq (\displaystyle\max_{j,k \in S}\eta_{jk}^{\textbf{u}})\mathbb{E}_i^{\textbf{u}}N\,.
\end{eqnarray*}
Using (A6) and the fact that the expected sojourn times in each state is finite it follows that
$$\displaystyle \sup_{\textbf{u}\in \mathcal{U}}(\displaystyle\max_{j,k \in S}\eta_{jk}^{\textbf{u}})< \infty .$$

\noindent Note that for the above the finiteness of the state
space is crucial. Hence the desired result follows by
$\eqref{frt}$.
\end{proof}

\begin{lem}\label{ubounded}For $\alpha > 0$, let $h_{\alpha}(i) = V_{\alpha}(i) - V_{\alpha}(0)$.
Then the family $\{h_{\alpha}\}_{\alpha>0}$ is uniformly bounded.
\end{lem}
\begin{proof}Let $K$ be a constant such that $\displaystyle \max_i \sup_{\textup{\textbf{u}}}\mathbb{E}^{\textup{\textbf{u}}}_{i,0}\tau_0 < K$. If $\textbf{u}_{\alpha}^*$ denotes the
optimal policy for the $\alpha-$discounted case then we have,
\begin{align*}V_{\alpha}(i)=&\mathbb{E}_{i,0}^{\textbf{u}_{\alpha}^*} \biggl[\int_0^{\tau_0}
e^{-\alpha t}\tilde{c}(X_t,Y_t, \textbf{u}^*_{\alpha}(X_t,Y_t))dt \\&+ \int_{\tau_0}^{\infty}e^{-\alpha t}\tilde{c}(X_t,Y_t, \textbf{u}^*_{\alpha}(X_t,Y_t))dt\biggr]\\
&\leq \tilde{C}K+ \mathbb{E}_{i,0}^{\textbf{u}_{\alpha}^*}e^{-\alpha \tau_0}V_{\alpha}(0)\\
&\leq \tilde{C}K+V_{\alpha}(0)\;.
\end{align*}Again,
$$\mathbb{E}_{i,0}^{\textbf{u}_{\alpha}^*}e^{-\alpha \tau_0}V_{\alpha}(0) \leq V_{\alpha}(i)$$
Thus,
\begin{align*}V_{\alpha}(0) & \leq V_{\alpha}(i)+\bigl(1-\mathbb{E}_{i,0}^{\textbf{u}_{\alpha}^*}e^{-\alpha \tau_0}\bigr)V_{\alpha}(0) \\
&\leq V_{\alpha}(i)+\bigl(1-e^{-\alpha K}\bigr)\frac{\tilde{C}}{\alpha}\\
& \leq V_{\alpha}(i) + K\tilde{C}\;.
\end{align*}The second inequality follows from Jensen's inequality.

\noindent Thus we have $$|h_{\alpha}(i)|\leq K\tilde{C}\,.$$
\end{proof}

\begin{thm}Under (A1)-(A6), the equation $\eqref{smdp}$ has a solution $(h,g)$ where $h$ is a bounded function and $g$ is a scalar.
\end{thm}
\begin{proof}Let $\tilde{h}_{\alpha}(i)=\tilde{V}_{\alpha}(i)-\tilde{V}_{\alpha}(0)$. Then by Lemma $\ref{ubounded}$ and
 $\eqref{equi}$, it follows that the family
$\{\tilde{h}_{\alpha}(i)\}$ is uniformly bounded. Therefore there
exists a sequence $\alpha_n \rightarrow 0$ such that
$$g=\displaystyle\lim_{\alpha_n \rightarrow 0}\alpha _n \tilde{V}_{\alpha_n}(0)$$
$$h(i)=\displaystyle\lim_{\alpha_n \rightarrow 0}\tilde{h}_{\alpha_n}(i)$$ where $h$ is a bounded function.
Now using standard arguments \cite{ross1992applied}, it can be
shown that the pair $(g,h)$ satisfies $\eqref{smdp}$.
\end{proof}
\begin{rem}If $\{X_t^{\textbf{u}}\}$ is irreducible for each $\textbf{u} \in \mathcal{U}$, i.e., if the embedded Markov chain is irreducible then
$$\displaystyle \limsup_{n \rightarrow \infty}\frac{\mathbb{E}_{i,0}^{\textbf{u}}\int_0^{T_n}
\int_U c(X_t,Y_t,u)\textbf{u}(X_t,Y_t)(du)dt}{\mathbb{E}_{i,0}^{\textbf{u}}T_n}=
\displaystyle \limsup_{T \rightarrow \infty}\frac{1}{T}\mathbb{E}_{i,0}^{\textbf{u}}\int_0^T \int_Uc(X_t,Y_t,u)\textbf{u}(X_t,Y_t)(du)dt\,.$$
Thus if the irreducibility assumption holds, then $g$ of the above theorem satisfies
$$g=\displaystyle \limsup_{T \rightarrow \infty}\frac{1}{T}\mathbb{E}_{i,0}^{\textbf{u}}\int_0^T \int_Uc(X_t,Y_t,u)\textbf{u}(X_t,Y_t)(du)dt\,.$$
\end{rem}

\section{Conclusions}
We have studied  optimal control problems for a class Markov
processes with age dependent transitions rates which subsumes
semi-Markov decision processes with the holding time distributions
having densities. We have allowed control actions between jumps
based on the age of the process. We have constructed an equivalent
SMDP which yields the relevant results for the original problem. A
standard approach towards solving an optimal control problem  is
via the HJB equation. In our problem the HJB equation for the
discounted cost case is given by
\begin{eqnarray}\label{hjb}\frac{d\varphi(i,y)}{dy}+ \inf_{u}\large[c(i,y,u)+ \sum_{j \neq i}\lambda_{ij}(y,u)
\large\{\varphi(j,0)-\varphi(i,y)\}\large]=\alpha
\varphi(i,y)\end{eqnarray}on $S \times [0,\infty)$. One important
difficulty in handing with this differential equation is that it
is non-local. It can be be shown via contraction principle
argument that when $\alpha > 2M$, the value function $V_{\alpha}$
is the unique bounded, smooth solution of \eqref{hjb}. In this
case the infimum in \eqref{hjb} is realised at a stationary
deterministic (non-relaxed) control which is optimal for the
$\alpha$-discounted cost criteria. But we have not been able to
establish the existence of a solution to \eqref{hjb} when $\alpha
\leq 2M$. Because we have not been able to solve the discounted
case HJB for smaller values of $\alpha$, we could not pursue the
vanishing discount approach in finding a solution to the HJB
equation for the average optimal case. In our problem the HJB
equation for the average optimal case is given by
\begin{eqnarray}\label{aohjb}
\rho= \frac{dh(i,y)}{dy}+\inf_{u}\large[c(i,y,u)+ \sum_{j \neq
i}\lambda_{ij}(y,u)\large\{h(j,0)-h(i,y)\}\large]\,.\end{eqnarray}It
would be interesting to investigate an appropriate solution of
\ref{aohjb} to study the average optimal case.

Finally, in this paper we have assumed that the jump rates and the
cost function are bounded. If the jump rates are unbounded but
satisfy a certain growth rate, then following the arguments in
Chapter $8$, Section $3$ in \cite{Ethier}, one can show that the
controlled martingale problem for the operator
\begin{eqnarray}\mathcal{A}^{\textbf{u}}f(i,y)= \frac{\partial f}{\partial y}(i,y)+
 \sum_{j \neq i}\tilde{\lambda}_{ij}(t,y,\textbf{u}(i,y))[f(j,0)-f(i,y)] \;.
\end{eqnarray} is well-posed. For an unbounded cost, with an appropriate growth rate it may be possible to work in the space of
continuous functions with weighted norms as in \cite{hgbook},
\cite{hernandez} to derive analogous results.

\end{document}